\documentclass[12pt]{amsart}
\usepackage{amsmath,mathrsfs,cite}

\usepackage[ps2pdf,colorlinks=true,urlcolor=blue,
citecolor=red,linkcolor=blue,linktocpage,pdfpagelabels,bookmarksnumbered,bookmarksopen]{hyperref}
\usepackage[english]{babel}
\usepackage[left=2.68cm,right=2.68cm,top=2.5cm,bottom=2.5cm]{geometry}

\numberwithin{equation}{section}

\newcommand{\R}{{\mathbb R}}
\newcommand{\Dis}{{\mathbb D}}
\newcommand{\Sf}{{\mathbb S}}

\newcommand{\eps}{\varepsilon}

\newcommand{\dom}[1]{{\rm dom}(#1)}
\newcommand{\ws}[1]{|d#1|}

\renewcommand{\theta}{\vartheta}

\numberwithin{equation}{section}
\newtheorem{theorem}{Theorem}[section]
\newtheorem{proposition}[theorem]{Proposition}
\newtheorem{lemma}[theorem]{Lemma}
\newtheorem{remark}[theorem]{Remark}
\newtheorem{corollary}[theorem]{Corollary}
\newtheorem{definition}[theorem]{Definition}
\theoremstyle{definition}

\title{On Struwe-Jeanjean-Toland monotonicity trick}

\author{Marco Squassina}
\address{Dipartimento di Informatica
\newline\indent
Universit\`a degli Studi di Verona
\newline\indent
C\'a Vignal 2, Strada Le Grazie 15, I-37134 Verona, Italy}
\email{marco.squassina@univr.it}

\thanks{Research supported by PRIN: {\em Metodi Variazionali e Topologici
nello Studio di Fenomeni non Lineari}}

\begin{document}
	
\date{\today}



\begin{abstract}
The abstract version of Struwe's monotonicity trick developed by Jeanjean and Jeanjean-Toland for functionals
depending on a real parameter is strengthened in the sense that it provides, for almost 
every value of the parameter, the existence of a bounded almost symmetric Palais-Smale sequence at the Mountain Pass level, 
whenever a mild symmetry assumption is set on the energy functional. Besides, all the machinery
is extended to the case of continuous functionals on Banach spaces, in the framework of non-smooth critical point theory.
\end{abstract}
\maketitle

\section{Introduction}
It is known that there are situations, often related to physically relevant PDEs associated with an energy functional $f$,
where it is particularly difficult to establish the {\em boundedness} of {\em Palais-Smale} sequences for $f$. 
In order to overcome this difficulty, {\sc Struwe} \cite{ambstruw,
struwe,struwe1,struwe2,struwem,struwtarant} introduced, around 1988, the
so-called {\em monotonicity trick}. In solving important problems he showed
how the fact that the underlying functional enjoys some monotonicity
properties could be used in order to derive a bounded Palais-Smale sequence. 
About ten years later, it was shown by {\sc Jeanjean} \cite{Jj} that it was possible to formulate a general abstract statement based
upon the monotonicity trick. This contribution is of particular relevance since it provides a ready-to-use machinery in order to tackle
variational PDE's for which the Palais-Smale condition is hard to manage. The principle says, essentially, that 
given a family of $C^1$ smooth functionals $f(\lambda;\cdot)$ satisfying a uniform Mountain Pass geometry and
monotonically depending on the parameter $\lambda$, then the almost everywhere
differentiability of the Mountain Pass value $c(\lambda)$ induces the existence of a bounded Palais-Smale sequence for $f(\lambda;\cdot)$ 
for almost every $\lambda$ in the interval $\Lambda$ where the family is defined. 
This property cannot be improved in general, in light of a counterexample due to {\sc Brezis-Nirenberg} (cf.~\cite{Jj}), which
shows that in some cases there may exist values of $\lambda$ for which {\em any} Palais-Smale sequence 
at the level $c(\lambda)$ is {\em unbounded}. 
Similar phenomena are know to occur in the study of periodic solutions
to Hamiltonian systems (cf.~\cite{ginz,herm}). We refer the reader to~\cite{Jj} for applications 
to a Landesman-Lazer type problem on $\R^N$, to~\cite{Jjbif} for a use in bifurcation analysis and, finally, 
to~\cite{Jtanak,wilzou} where the technique was used to investigate some classes of nonlinear Schr\"odinger equations. 
An important extension was done in {\sc Jeanjean-Toland} \cite{JT}, where it became clear that for the
monotonicity trick to hold true, actually, neither the monotonicity of the family 
$f(\lambda;\cdot)$ nor the differentiability of its related Mountain Pass value $c(\lambda)$ are needed. 
Although for the majority of concrete problems the dependence of the family $f(\lambda;\cdot)$ upon $\lambda$ is monotone,
in \cite{JT} some situations can be covered in the case where the family $f(\lambda;\cdot)$ has the form
$J(\lambda;u)-\lambda I(u)$, where $I,J:X\to\R$ are $C^1$ functionals with suitable structural assumptions.
The abstract results of~\cite{Jj} have also been extended e.g.~by {\sc Szulkin-Zou}, {\sc Zou-Schechter} and {\sc Schechter}
to other minimax structures with a nice impact on PDEs (see \cite{szuzou,schezou1,schezou2}, 
the monograph~\cite{zouschebook} and references therein).
\vskip2pt
The scope of the present manuscript is twofold. 
\vskip2pt
As a main goal, in Theorem~\ref{maintth1} and Corollaries~\ref{maintth1-cor}-\ref{maincor2}, 
we improve the abstract $(C^1)$ version of {\sc Jeanjean-Toland} \cite{JT} machinery 
in the sense that, up to a set of null measure, for each value of the parameter $\lambda$ we can find a bounded Palais-Smale
sequence $(u_h)\subset X$ for $f$ at the Mountain Pass level $c(\lambda)$ which is {\em almost symmetric}, in the sense that
\begin{equation}
	\label{quasisymmcond}
\|u_h-u_h^*\|_V\to 0,\quad\text{as $h\to\infty$,}
\end{equation}
where $V$ is a Banach space with $X\hookrightarrow V$ continuously, whenever a symmetry assumption,
satisfied for a wide range of concrete cases, is assumed on $f$.\ Such sequences will be called
$(SBPS)_{c(\lambda)}$-sequences (see Definition~\ref{defps3}). Here $u^*$ denotes an abstract symmetrization of $u$ (according to \cite{vansch}), 
for instance it can be the classical Schwarz symmetrization when we take $X=W^{1,p}_0(\Omega)$ for $\Omega$ either a ball in $\R^N$
or the whole $\R^N$. If in addition the functional
satisfies the {\em symmetric bounded Palais-Smale condition}, in short $(SBPS)_{c(\lambda)}$,
at the limit one finds a symmetric Mountain Pass 
critical point. We stress that, in various situations (like {\em noncompact problems}) 
showing that, for some level $c\in\R$, a functional satisfies $(SBPS)_{c}$ 
is possible and quite direct (cf.~\cite[proof of Theorem 4.5]{vansch}) while 
the Palais-Smale condition, in general, fails \cite[Theorem 8.4]{willembook}.
In fact, handling a $(SBPS)_{c}$ sequence allows to exploit the {\em compact} embeddings of
a spaces of symmetric functions into a suitable Banach space (see e.g.~\cite[Section 1.5]{willembook}). 
In some sense, as pointed out in \cite{vansch} as well, the additional information 
about the {\em almost symmetry} of the Palais-Smale sequence
provides an alternative to {\em concentration-compactness} \cite{lions1,lions2}.
Notice that it is meant that the energy functional is {\em not} apriori restricted to a space of symmetric functions
as usually done in applying the well known Palais' {\em symmetric criticality principle} \cite{palais-scp}, recently 
extended by the author \cite{sqradII} to a nonsmooth framework (see also~\cite{sqradI}).
\vskip2pt
As a second goal, we shall extend the monotonicity trick to the class of {\em continuous} functionals, in the framework
of non-smooth critical point theory. If $\Omega\subset\R^2$ is bounded, applications 
of the monotonicity trick (see \cite{meanfmontrick,struwtarant})
have been provided for the problem $-\Delta u=\lambda (\int_\Omega e^u)^{-1}e^u$
with Dirichlet boundary conditions, which is naturally associated with 
the $C^1$ functional $f(\lambda;\cdot):H^1_0(\Omega)\to\R$ defined by 
$$
f(\lambda;u)=\frac{1}{2}\int_\Omega |Du|^2-\lambda \log\Big(\frac{1}{|\Omega|}\int_\Omega e^u\Big).
$$
The previously mentioned equation can be also studied on a Riemannian manifold $(M,g)$, in which case the Laplace operator
is replaced by the Laplace-Beltrami operator $\Delta_g$. More generally, following some indications coming
from differential geometry \cite{uhlen}, one can think of equations on a manifold, associated with functional
having a kinetic part of of the form:
$$
\int_M j(x,s,Ds)d\mu_M(x),\quad    j(x,s(x),Ds(x))=G_{ij}(x,s(x))D_is(x)D_js(x),
$$
which, due to the explicit dependence upon $s(x)$ in the integrand, are non-smooth (not even locally Lipschitz).
In a similar fashion, in the context of diffusion processes such as heat-conduction, explicit dependence of the 
$s(x)$ in the kinetic part of the functional
has to be expected in the case of non-homogeneous and non-isotropic materials (cf.~\cite{ash,tritt}).
Therefore, it is reasonable to think that some situations can occur where the functional $f(\lambda;\cdot)$ under study
is of the form $J(\lambda;u)-\lambda I(u)$, where $J(\lambda;\cdot):X\to\R$ are
merely continuous (or even less regular) functionals while $I(\lambda;\cdot):X\to\R$ are $C^1$ functionals. 
In order to deal with this level of generality, we will use a suitable non-smooth critical point theory, now 
well-established, developed about twenty years ago (see e.g.~\cite{corv,cdm,dm}). 
\vskip4pt
The plan of the paper is as follows. 
\vskip4pt
In Section~\ref{preliminari}, we recall a few notions and
results from non-smooth critical point theory and symmetrization theory. In Section~\ref{risultati}
we state and comment the main result of the paper (Theorem~\ref{maintth1}) as well as two useful 
consequences (Corollaries~\ref{maintth1-cor} and \ref{maincor2}). Finally, in 
Section~\ref{prove}, we provide the proofs of the results.

\section{Some preliminary facts}
\label{preliminari}
In this section we recall abstract notions
and results from non-smooth critical point and symmetrization theories
that will be used in the proof of the main results.

\subsection{Tools from symmetrization theory}
We refer to~\cite{vansch} and references therein.

\subsubsection{Abstract symmetrization}
\label{abstractsymmetrrr}
Let $X$ and $V$ be two Banach spaces and $S\subseteq X$.
We consider two maps $*:S\to S$, $u\mapsto u^*$ 
({\em symmetrization map}) and $h:S\times {\mathcal H}_*\to S$,
$(u,H)\mapsto u^H$ ({\em polarization map}), where ${\mathcal H}_*$ 
is a path-connected topological space. We assume the following conditions:
\begin{enumerate}
 \item $X$ is continuously embedded in $V$;
 \item $h$ is a continuous mapping;
\item for each $u\in S$ and $H\in {\mathcal H}_*$ it holds $(u^*)^H=(u^H)^*=u^*$ and $u^{HH}=u^H$;
\item there exists $(H_m)\subset {\mathcal H}_*$ such that, for $u\in S$, $u^{H_1\cdots H_m}$ converges
to $u^*$ in $V$;
\item for every $u,v\in S$ and $H\in {\mathcal H}_*$ it holds
$\|u^H-v^H\|_V\leq \|u-v\|_V$.
\end{enumerate}
Furthermore $*:S\to V$ can be extended to the whole space $X$ by 
setting $u^*:=(\Theta(u))^*$ for all $u\in X$, where $\Theta:(X,\|\cdot\|_V)\to (S,\|\cdot\|_V)$ is
a Lipschitz function such that $\Theta|_{S}={\rm Id}|_{S}$. 
It is readily seen that, within this framework, there exists $C_\Theta>0$ such that 
\begin{equation}
	\label{misspropp} 
\|u^*-v^*\|_V\leq C_\Theta\|u-v\|_V,\quad\text{for all $u,v\in X$}.
\end{equation}

\subsubsection{Concrete polarization}
A subset $H$ of $\R^N$ is called a polarizer if it is a closed affine half-space
of $\R^N$, namely the set of points $x$ which satisfy $\alpha\cdot x\leq \beta$
for some $\alpha\in \R^N$ and $\beta\in\R$ with $|\alpha|=1$. Given $x$ in $\R^N$
and a polarizer $H$ the reflection of $x$ with respect to the boundary of $H$ is
denoted by $x_H$. The polarization of a function $u:\R^N\to\R^+$ by a polarizer $H$
is the function $u^H:\R^N\to\R^+$ defined by
\begin{equation}
 \label{polarizationdef}
u^H(x)=
\begin{cases}
 \max\{u(x),u(x_H)\}, & \text{if $x\in H$} \\
 \min\{u(x),u(x_H)\}, & \text{if $x\in \R^N\setminus H$.} \\
\end{cases}
\end{equation}
The polarization $C^H\subset\R^N$ of a set $C\subset\R^N$ is 
defined as the unique set which satisfies $\chi_{C^H}=(\chi_C)^H$,
where $\chi$ denotes the characteristic function. 
The polarization $u^H$ of a positive function $u$ defined on $C\subset \R^N$
is the restriction to $C^H$ of the polarization of the extension $\tilde u:\R^N\to\R^+$ of 
$u$ by zero outside $C$. The polarization of a function which may change sign is defined
by $u^H:=|u|^H$, for any given polarizer $H$.

\subsubsection{Concrete symmetrization}
\label{concresectsym}
The Schwarz symmetrization of a set $C\subset \R^N$ is the unique open ball 
centered at the origin $C^*$ such that ${\mathcal L}^N(C^*)={\mathcal L}^N(C)$, being ${\mathcal L}^N$ the
$N$-dimensional outer Lebesgue measure. If the measure of $C$ is zero we set $C^*=\emptyset$,
while if the measure of $C$ is not finite we put $C^*=\R^N$. 
A measurable function $u$ is admissible for the Schwarz symmetrization if it is nonnegative and, for every $\eps>0$,
the Lebesgue measure of $\{u>\eps\}$ is finite. The Schwarz symmetrization
of an admissible function $u:C\to\R^+$ is the unique function $u^*:C^*\to\R^+$ such that,
for all $t\in\R$, it holds $\{u^*>t\}=\{u>t\}^*$. Considering the extension 
$\tilde u:\R^N\to\R^+$ of $u$ by zero outside $C$, then
$u^*=(\tilde u)^*|_{C^*}$ and $(\tilde u)^*|_{\R^N\setminus C^*}=0$. 
The symmetrization for possibly changing sign $u$ can be the defined by $u^*:=|u|^*$. Let
${\mathcal H}_*=\{H\in{\mathcal H}: 0\in H\}$ and
$\Omega$ a ball in $\R^N$ or the whole space $\R^N$. Then let us set either 
$$
X=W^{1,p}_0(\Omega),\quad 
S=W^{1,p}_0(\Omega,\R^+),\quad
V=L^p\cap L^{p^*}(\Omega)
$$ 
or 
$$
X=S=W^{1,p}_0(\Omega),\quad
V=L^p\cap L^{p^*}(\Omega),\quad 
u^H:=|u|^H,\, u^*:=|u|^*.
$$
Then (1)-(5) in the abstract framework are satisfied (cf.~e.g.~\cite{vansch}).

\subsubsection{Symmetric approximation of curves}
In the proof of the main result, in order to overcome the lack (in general, cf.~\cite{almlieb}) of continuity
of the symmetrization map $u\mapsto u^*$, we shall need an approximation tool for continuous 
curves \cite[Proposition 3.1]{vansch}, that we state adapted to 
a particular framework. In the following, $\Dis$ and $\Sf$ will 
always denote the closed unit ball and sphere in $\R^m$ with $m\geq 1$.

\begin{proposition}
\label{approxxres}
Let $X$ and $V$ be two Banach spaces, $S\subseteq X$, $*$ and ${\mathcal H}_*$ 
which satisfy the requirements of the abstract symmetrization framework~\eqref{abstractsymmetrrr}. Let $M$ be 
a closed subset of $\Dis$, disjoint from $\Sf$, and $\gamma \in C(\Dis,X)$. Let $H_0\in {\mathcal H}_*$
and $\gamma(\Dis)\subset S$. Then, for every $\delta>0$, there exists a curve 
$\tilde\gamma\in C(\Dis,X)$ such that 
$$
\|\tilde\gamma(\tau)-\gamma(\tau)^*\|_V\leq \delta,\,\,\quad\text{for all $\tau\in M$},
$$
$\tilde\gamma (\tau)=\gamma(\tau)^{H_0H_1\cdots H_{[\theta]}H_\theta}$
for all $\tau\in \Dis$, with $\theta\geq 0$, $H_s\in {\mathcal H}_*$ for $s\geq 0$,
$\tilde\gamma(\tau)=\gamma(\tau)^{H_0}$ for all $\tau\in \Sf$. Here $[\theta]$ denotes
the largest integer less than or equal to $\theta$ and the polarizer $H_\theta$ is 
introduced in~\cite[Proposition 2.7]{vansch}.
\end{proposition}

\subsection{Tools from non-smooth critical point theory}
For definitions and notions in this section, we refer the reader to~\cite{cdm,dm} and the references therein.
In the following $(X,d)$ will denote a metric space and $B(u,\delta)$ the open ball in $X$ of center
$u$ and of radius $\delta$. 

\begin{definition}\label{defslope}
Let $f:X \to \R $ be a continuous function,
and $u\in X$. We denote by
$|df|(u)$ the supremum of the real numbers $ \sigma$ in
$[0,\infty)$ such that there exist $ \delta >0$
and a continuous map
$
{\mathcal H}\,:\,B(u, \delta) \times[ 0, \delta]  \to X,
$
such that, for every $v$ in $B(u,\delta) $, and for every
$t$ in $[0,\delta]$ it results
\begin{equation*}
d({\mathcal H}(v,t),v) \leq t,\qquad
f({\mathcal H}(v,t)) \leq f(v)-\sigma t.
\end{equation*}
The extended real number $|df|(u)$ is called the weak
slope of $f$ at $u$.
\end{definition}

We recall from~\cite{dm} a well known fact.

\begin{proposition}
	\label{lowersecslope}
	Let $f:X\to\R$ be a continuous functional. If $(u_h)\subset X$ is a sequence converging to $u$ in $X$, then
	$$
	|df|(u)\leq \liminf_h |df|(u_h).
	$$
\end{proposition}

The next result establishes the connection between the weak
slope of a function $f$ and its differential $df(u)$, in the case where $f$ is 
of class $C^1$, see~\cite[Corollary 2.12]{dm}.

\begin{proposition}
If $X$ is an open subset of a normed space $E$ and $f$ is a function of class $C^1$
on $X$, then $|df|(u)=\|df(u)\|$ for every $u\in X$.
\end{proposition}

We recall from \cite{corv} the following Quantitative Deformation Lemma (cf.~\cite[Theorem 2.3]{corv})

\begin{lemma}
\label{theorlem}
Assume that $X$ is a complete metric space and $f:X\to\R$ is a continuous functional,
$c\in\R$, $A$ is a closed subset of $X$ and $\delta,\sigma>0$ are such that
$$
\text{$c-2\delta\leq f(u)\leq c+2\delta$\, and \, $d(u,A)\leq\frac{\delta}{\sigma}$}\quad
\Longrightarrow\quad |df|(u)>2\sigma.
$$
Then there exists a continuous map $\eta:X\times[0,1]\to X$ such that  
$$ 
d(\eta(u,t),u)\leq \frac{\delta}{\sigma}t,\qquad
\eta(u,t)\neq u \,\,\,\Longrightarrow\,\,\,  f(\eta(u,t))<f(u),
$$ 
and
$$  
u\in A,\,\,\, c-\delta\leq f(u)\leq c+\delta\quad\Longrightarrow\quad
f(\eta(u,t))\leq f(u)-(f(u)-c+\delta) t,
$$
for every $u\in X$ and $t\in[0,1]$.
\end{lemma}

The previous notions allow us to give the next definition.
\begin{definition}
We say that $u\in\dom{f}$ is a critical
point of $f$ if $\ws{f}(u)=0$.
We say that $c\in\R$ is a critical value of
$f$ if there is a critical point $u\in\dom{f}$
of $f$ with $f(u)=c$. 
\end{definition}

Finally, we consider a useful notion of (almost) symmetry for Palais-Smale sequences. 

\begin{definition}\label{defps3}
Let $(X,\|\cdot\|)$ and $(V,\|\cdot\|_V)$ be Banach spaces which are compatible 
with the abstract symmetrization framework~\ref{abstractsymmetrrr}.
We say that $(u_n)\subset X$ is a {\rm Symmetric Bounded Palais-Smale sequence} at level $c\in\R$ ($(SBPS)_c$-sequence) if $(u_n)$ is bounded in $X$,
$\ws{f}(u_n) \to 0$, $f(u_n) \to c$ and, in addition, 
$$
\lim_{n\to\infty}\|u_n-u_n^*\|_V=0.
$$ 
We say that $f$ satisfies the {\rm Symmetric Bounded Palais-Smale condition} at level $c$ ($(SBPS)_c$
in short), if every $(SBPS)_c$ sequence admits a subsequence converging in $X$.
\end{definition}

\section{The results}
\label{risultati}
In this section we state and prove the main results of the paper.

\subsection{Assumptions}
Let $(X,\|\cdot\|)$ and $(V,\|\cdot\|_V)$ be two real Banach spaces, $S\subseteq X$, $*$ and ${\mathcal H}_*$ 
which satisfy the requirements of the abstract symmetrization framework~\eqref{abstractsymmetrrr}. We consider the 
following assumptions:
\noindent
\vskip5pt
\noindent
$\boldsymbol{({\mathcal H}_1)}$ Let $\Lambda\subset\R$ be a compact interval and 
$$
f:\Lambda\times X\to\R,
$$
a family of functionals such that, for all $\lambda\in \Lambda$, $f(\lambda;\cdot)$ is continuous. 
\noindent
\vskip4pt
\noindent
$\boldsymbol{({\mathcal H}_2)}$ If $\Gamma_0\subset C(\Sf;X)$, then for all $\lambda\in \Lambda$:
$$
c(\lambda)>a(\lambda),\qquad a(\lambda):=\sup_{\gamma_0\in\Gamma_0}\sup_{\tau\in\Sf}f(\lambda;\gamma_0(\tau)),
$$
where $c(\lambda)$ is the Mountain Pass values defined by
\begin{equation}
\label{mountainpasdef}
c(\lambda):=\inf_{\gamma\in \Gamma}\sup_{t\in \Dis} f(\lambda;\gamma(t)),\qquad
\Gamma:=\{\gamma\in C(\Dis,X):\gamma|_{\Sf}\in\Gamma_0\}\quad (\Gamma\neq\emptyset).
\end{equation}
\noindent
\vskip4pt
\noindent
$\boldsymbol{({\mathcal H}_3)}$ For every sequence $(\lambda_h,u_h)\subset \Lambda\times X$ 
with $(\lambda_h)$ strictly increasing and converging to $\lambda$ 
for which there exists $C\in\R$ with
$$
f(\lambda_h;u_h)\leq C, \,\,\quad
-f(\lambda;u_h)\leq C, \,\,\quad
\frac{f(\lambda_h;u_h)-f(\lambda;u_h)}{\lambda-\lambda_h}\leq C,\qquad
\text{for all $h\geq 1$},
$$
then $\|u_h\|\leq {\mathcal M}$ for some number ${\mathcal M}={\mathcal M}(C)\geq 0$ and all $h\geq 1$ and, for every $\eps>0$ 
$$
f(\lambda,u_h)\leq f(\lambda_h;u_h)+\eps,  \qquad
\text{for all $h\geq 1$ sufficiently large}.
$$
\noindent
\vskip4pt
\noindent
$\boldsymbol{({\mathcal H}_4)}$ For all $\gamma\in\Gamma$
there are $\hat\gamma\in\Gamma$ and $H_0\in {\mathcal H}_*$ with 
$\hat\gamma(\Dis)\subset S$ and $\hat\gamma|_{\Sf}^{H_0} \in \Gamma_0$ such that
$$
f(\lambda;\hat\gamma(t))\leq f(\lambda;\gamma(t)),\qquad
\text{for all $t\in \Dis$ and $\lambda\in\Lambda$.}
$$
Moreover, for all $\lambda\in\Lambda$,
$$
f(\lambda;u^H)\leq f(\lambda;u),\qquad
\text{for all $H\in {\mathcal H}_*$ and $u\in S$.}
$$
\subsubsection{Some remarks on the assumptions}
Concerning $\boldsymbol{({\mathcal H}_2)}$, it is a uniform Mountain Pass geometry for the family
of functions $\{f(\lambda;\cdot)\}_{\lambda\in\Lambda}$. In the minimax principle one could also allow a more general situation
where $\Gamma=\Gamma(\lambda)$ depends on $\lambda$. On the other hand, in this case one needs some monotonicity property on $\Gamma(\lambda)$, 
for instance $\Gamma(\lambda)\subseteq\Gamma(\mu)$, for every $\lambda<\mu$. One can think for instance to the two important (classical) cases:
\begin{align}
	\label{classGamma}
\Gamma&=\{\gamma\in C([0,1];X):\gamma(0)=0,\,\, \gamma(1)=v\},
\quad\text{with $f(\lambda;v)<0$, for all $\lambda\in\Lambda$}, \\
\label{classGamma-bis}
\Gamma(\lambda) &=\{\gamma\in C([0,1];X):\gamma(0)=0,\,\, f(\lambda;\gamma(1))<0\},
\end{align}
corresponding in $\boldsymbol{({\mathcal H}_2)}$ to the choice $\Dis=[0,1]$, $\Sf=\{0,1\}$ and $\Gamma_0=\{0,v\}$.
Assuming that the map $\lambda\mapsto f(\lambda;\cdot)$ is decreasing, then $\lambda\mapsto \Gamma(\lambda)$ is increasing. 
The choice of~\eqref{classGamma} for the construction of $c(\lambda)$
is probably the most classical and widely used, and it is precisely the minimaxing
family of curves used in~\cite{Jj,JT}.
Concerning condition $\boldsymbol{({\mathcal H}_3)}$, it is precisely 
the one originally formulated by Jeanjean and Toland \cite{JT} and it aims to select
a particular sequence $(\gamma_n)$ of curves in $\Gamma$, which enjoy some good properties. As pointed out in \cite[Example 2.1]{JT},
functionals of the form $f(\lambda;u)=A(\lambda;u)-\lambda B(u)$ satisfy $\boldsymbol{({\mathcal H}_3)}$, under suitable assumptions.
If in addition $A$ is independent of $\lambda$ the last property in $\boldsymbol{({\mathcal H}_3)}$ automatically holds and the boundedness
of $(u_h)$ follows by the coerciveness of either $A(u)$ or $B(u)$ (cf.~\cite{Jj}).
Finally, compared with~\cite{JT}, $\boldsymbol{({\mathcal H}_4)}$ is the new additional assumption and it constitutes the natural
link with symmetrization theory. We stress that it is fulfilled in a broad range of 
meaningful cases (see \cite{vansch,sqradI}). 
In the Sobolev case $S=W^{1,p}_0(\Omega,\R^+)\subset W^{1,p}_0(\Omega)=X$ (cf.\ sections~\ref{abstractsymmetrrr}-\ref{concresectsym}), 
choosing the family \eqref{classGamma} one uses
a function $v\geq 0$ with $v^{H_0}=v$ and $f(\lambda;v)<0$ for some $H_0\in {\mathcal H}_*$ and all $\lambda\in\Lambda$. Hence, 
if $\gamma\in\Gamma$ and $\hat\gamma(t):=|\gamma(t)|\in S$, it follows that
$f(\lambda;\hat\gamma(t))\leq f(\lambda;\gamma(t))$ for all $t\in [0,1]$ and $\lambda\in\Lambda$
if for instance $f(\lambda,|\cdot|)\leq f(\lambda,\cdot)$. Moreover,
$\hat\gamma(0)^{H_0}=0\in\Gamma_0$ and $\hat\gamma(1)^{H_0}=v^{H_0}=v\in\Gamma_0$. Choosing instead
the family \eqref{classGamma-bis}, if we fix some $H_0\in {\mathcal H}_*$, we have 
$f(\lambda,\hat\gamma(1)^{H_0})=f(\lambda,|\gamma(1)|^{H_0})\leq f(\lambda,|\gamma(1)|)\leq f(\lambda,\gamma(1))<0$,
so that again $\hat\gamma(0)^{H_0},\hat\gamma(1)^{H_0}\in\Gamma_0$, as requested by the first part of $\boldsymbol{({\mathcal H}_4)}$.
Similar choices are made in the case where one takes $S=X=W^{1,p}_0(\Omega)$ (cf.\ sections~\ref{abstractsymmetrrr}-\ref{concresectsym}).

\subsection{Statements}
Under $\boldsymbol{({\mathcal H}_1)}$-$\boldsymbol{({\mathcal H}_4)}$, we now state the main result of the paper.

\begin{theorem}
	\label{maintth1}
For almost every $\lambda\in \Lambda$, $f(\lambda;\cdot)$ possesses a $(SBPS)_{c(\lambda)}$-sequence. 
\end{theorem}

\noindent
In turn, under the same hypothesis, we also have the following
\begin{corollary}
	\label{maintth1-cor}
For almost every $\lambda\in \Lambda$, $f(\lambda;\cdot)$ 
possesses a critical point $u_\lambda\in X$ at level $c(\lambda)$ and with $u_\lambda=u^*_\lambda$,
provided it satisfies $(SBPS)_{c(\lambda)}$.
\end{corollary}

\noindent
Finally, inspired by Jeanjean~\cite[Corollary 1.2]{Jj}, we also have the following 
\begin{corollary}
	\label{maincor2}
Let $f(\lambda;\cdot)$ satisfy $(SBPS)_{c(\lambda)}$ for all $\lambda\in [1-\sigma,1]$, where $\sigma>0$. 
Then there exists a sequence $(\lambda_j,u_j)\subset [1-\sigma,1]\times X$ such 
that $\lambda_j\nearrow 1$ and for all $j\geq 1$
\begin{equation}
	\label{siamo1}
	f(\lambda_j;u_j)=c(\lambda_j),\,\,\qquad  |df(\lambda_j;\cdot)|(u_j)=0,\,\,\qquad u_j=u_j^*.
\end{equation}	
\end{corollary}

\noindent
The monotonicity trick in the form of \cite{Jj,JT} is thus improved in light of the {\em symmetry}
conclusions, as commented in the introduction, provided that a symmetry assumption on $f$, that 
is $\boldsymbol{({\mathcal H}_4)}$, is assumed. 
Corollary~\ref{maincor2} is particularly useful for the study of the functional $f(1;\cdot)$
on the basis of the properties of the {\em nearby} functionals $f(\lambda_j;\cdot)$, when
bounded Palais-Smale sequences of $f(\lambda;\cdot)$ are precompact for 
any $\lambda\in [1-\sigma,1]$ (in particular for $\lambda=1$).
In fact, it is expected that starting from~\eqref{siamo1} (which imply, in a Sobolev functional framework, that $u_j$ 
is a symmetric weak solution of an elliptic PDE, possibly in a suitable generalized sense, and thus very likely it fulfills
{\em extra qualitative properties}) one can deduce
$$
\sup_{j\geq 1}\|u_j\|<+\infty,
$$
and (in turn, by $u_j\rightharpoonup u$ in $X$ as $j\to\infty$, up to a subsequence),
$$
f(1;u_j)\to c(1),
\,\,\,\quad 
|df(1;\cdot)|(u_j)\to 0,\qquad \text{as $j\to\infty$},
$$
provided that $\lambda\mapsto c(\lambda)$ is left continuous (c.f.~\cite[lemma 2.3]{Jj} where this proved in the $C^1$ case), namely
$f(1;\cdot)$ admits a bounded Palais-Smale sequence at the Mountain Pass value $c(1)$. Therefore,
by the precompactness of bounded Palais-Smale sequence for $f(1;\cdot)$,
one can conclude that, $u_j\to u$ in $X$ as $j\to\infty$, so that
$f(1;\cdot)$ admits a nontrivial symmetric $(u=u^*)$ critical point $u$ at the Mountain Pass level $c(1)$. 
The symmetry, of course, follows by observing that (on account of~\eqref{siamo1} and~\eqref{misspropp})
$$
\|u-u^*\|_V\leq \|u-u_j\|_V+\|u_j-u^*_j\|_V+\|u^*_j-u^*\|_V\leq 2\|u-u_j\|_V\leq C\|u-u_j\|,
$$
yielding the desired conclusion, since $u_j\to u$ in $X$, as $j\to\infty$.
This line of argument has been successfully followed, without the additional symmetry property, 
in~\cite{Jtanak}, based upon the monotonicity trick of Jeanjean. 
Let us also mention that, in a more recent work \cite{azzpomp}, the authors restrict the functional
to a (Sobolev) space $X_r$ of {\em symmetric functions} in order to recover compactness. With
the improved version of the principle given by Corollary~\ref{maincor2}, the compactness would be recovered even working in the full space $X$, 
by crucially exploiting that $u_j=u_j^*$, see~\eqref{siamo1}, coming from the symmetry of the 
energy functional. Notice that, in~\cite{azzpomp}, the solution energy level is
$$
c_r(1)=\inf_{\gamma\in\Gamma_r}\sup_{t\in [0,1]} f(1;\gamma(t)),\quad \Gamma_r=\{\gamma\in C([0,1],X_r):\gamma(0)=0,\, f(1;\gamma(1))<0\},
$$
while using Corollary~\ref{maincor2}, we would find the solution at the level
$$
c(1)=\inf_{\gamma\in\Gamma}\sup_{t\in [0,1]} f(1;\gamma(t)),\quad \Gamma=\{\gamma\in C([0,1],X):\gamma(0)=0,\, f(1;\gamma(1))<0\},
$$
thus maintaining the {\em global minimizing property} of the Mountain Pass value.
\vskip2pt
Finally, Theorem~\ref{maintth1} and Corollary~\ref{maintth1-cor} hold for continuous functionals, in the framework of 
non-smooth critical point theory, allowing applications to quasi-linear PDEs (cf.~\cite{sq-monograph}).

\begin{remark}\rm
	\label{concreteexample}
A possible concrete framework where the abstract results can be applied is the following. 
Let $\Omega$ be either the whole $\R^N$ or the unit ball 
$B\subset\R^N$ centered at the origin, $N>p\geq 2$, $b>a>0$ and let 
$f:[a,b]\times W^{1,p}_0(\Omega)\to \R$ be the functional defined by
\begin{equation}
	\label{functtt}
f(\lambda;u):=\int_{\Omega}j(u,|\nabla u|)+\frac{1}{p}\int_\Omega |u|^p-\lambda\int_{\Omega}G(|x|,u),
\end{equation}
where $j\in C^1(\R\times\R^+)$, $t\mapsto j(s,t)$ is
strictly convex and increasing and there exist constants $\alpha_0,\alpha_1>0$ such that
$\alpha_0t^p\leq j(s,t)\leq \alpha_1 t^p$ for all $s\in\R$ and $t\in\R^+$ (see 
e.g.~the monograph~\cite{sq-monograph}).
Then, if the functions $G(|x|,s)=\int_0^s g(|x|,t)dt$, $g(|x|,s)$,
$j_s(s,t)$ and $j_t (s,t)$ satisfy suitable assumptions, 
conditions $\boldsymbol{({\mathcal H}_1)}$-$\boldsymbol{({\mathcal H}_4)}$
hold. In particular, it holds
$$
f(\lambda;u^H)\leq f(\lambda;u),\qquad
\text{for all $\lambda\in [a,b]$, any $H\in {\mathcal H}_*$ and $u\in W^{1,p}_0(\Omega)$},
$$
whenever $r\mapsto g(r,s)$ is decreasing, $j(|s|,t)\leq j(s,t)$ and $G(|x|,s)\leq G(|x|,|s|)$.
Notice that, if the growth of $j$ 
is weakened into $\alpha_0|\xi|^p\leq j(s,|\xi|)\leq \alpha(|s|)|\xi|^p$
for some possibly unbounded function $\alpha\in C(\R)$, then \eqref{functtt}
is merely lower semi-continuous from $W^{1,p}_0(\Omega)$ to $\R\cup\{+\infty\}$, for any $\lambda\in [a,b]$.
Statements~\ref{maintth1}-\ref{maincor2} are expected to hold also for {\em lower semi-continuous} 
functionals with suitable assumptions \cite{sqradI}.
On the other hand, in order to avoid excessive technicalities, we prefer to confine 
the analysis to the continuous case.
\end{remark}

\section{Proof of the results}
\label{prove}

Let $\lambda_0\in \Lambda$ be such that there
exist $Q(\lambda_0)\in\R$ and a strictly increasing sequence $(\lambda_h)$ converging to 
$\lambda_0$ as $h\to\infty$ and 
\begin{equation}
	\label{dung}
\frac{c(\lambda_h)-c(\lambda_0)}{\lambda_0-\lambda_h}\leq Q(\lambda_0),\,\,\quad\text{for all $h\geq 1$}.
\end{equation}
As pointed out in \cite{JT}, due to a result of Denjoy, the set $D\subseteq\Lambda$
of such points $\lambda_0$ is such that ${\mathcal L}^1(\Lambda\setminus D)=0$,
where ${\mathcal L}^1$ denotes the one-dimensional Lebesgue measure (for a $\lambda\in\Lambda\setminus D$
we would have the Dini's derivatives equal to $D^-c(\lambda_0)=D_-c(\lambda_0)=-\infty$, which is only possible on a set of zero measure).
\vskip2pt
First we formulate an improvement of~\cite[Lemma 2.1]{JT}, where the existence of suitable,
almost symmetric paths in $\Gamma$ enjoying special properties is obtained.  
We shall state the result for lower semi-continuous functionals.

\begin{lemma}
	\label{lemmaevolut}
	Assume that $f:\Lambda\times X\to\R\cup\{+\infty\}$ is a family of lower 
	semi-continuous functionals and that $\boldsymbol{({\mathcal H}_2)}$-$
	\boldsymbol{({\mathcal H}_4)}$ hold. Let $\lambda_0\in \Lambda$ be such that~\eqref{dung}
	is satisfied and let $(\lambda_h)$ be a related strictly increasing sequence converging to $\lambda_0$.
Then there exist $\bar h\geq 1$, two sequences of paths $(\gamma_h)_{h\geq \bar h},
(\tilde\gamma_h)_{h\geq \bar h}\subset\Gamma$ with $\gamma_h(\Dis),\tilde\gamma_h(\Dis)\subset S$,
a sequence $(M_h)_{h\geq \bar h}$ of nonempty closed subsets of $\Dis$, disjoint from $\Sf$, and a positive constant
${\mathcal M}(\lambda_0)$ such that
\begin{equation}
	\label{closssdd}
\|\tilde\gamma_h(t)-\gamma_h(t)^*\|_V\leq \lambda_0-\lambda_h,
\,\,\quad\text{for all $t\in M_h$},
\end{equation}
\begin{equation}
	\label{prima}
f(\lambda_0;\tilde\gamma_h(t))\geq c(\lambda_0)-\lambda_0+\lambda_h
\quad
\Longrightarrow
\quad
\|\tilde\gamma_h(t)\|\leq {\mathcal M}(\lambda_0),
\end{equation}
for all $h\geq \bar h$ and furthermore, for all $\eps>0$, it holds
\begin{equation}
	\label{seconda}
\sup_{t\in \Dis} f(\lambda_0;\tilde\gamma_h(t))\leq \sup_{t\in \Dis} f(\lambda_0;\gamma_h(t))\leq c(\lambda_0)+\eps
\end{equation}
for all $h\geq \bar h$ sufficiently large.
\end{lemma}

\begin{proof}
By the definition of $c(\lambda_h)$, as in~\cite[Lemma 2.1]{JT}, we can select a sequence 
$(\varrho_h)\subset \Gamma$ of curves such that, for all $h\geq 1$ large,
\begin{equation}
	\label{construct0}
\sup_{t\in \Dis} f(\lambda_h;\varrho_h(t))\leq c(\lambda_h)+\lambda_0-\lambda_h.
\end{equation}
In view $\boldsymbol{({\mathcal H}_4)}$, up to substituting $\varrho_h$ with $\hat\varrho_h$,
without loss of generality for all $h\geq 1$ we may assume that $\varrho_h(\Dis)\subset S$ and $\varrho_h|_{\Sf}^{H_0(h)}\in\Gamma_0$, 
for some polarizer $H_0(h)\in {\mathcal H}_*$.
Let now $\vartheta\in C(\Dis,\Dis)$ be defined by setting
$\vartheta(\tau)=\tau|\tau|^{-1}$ for all $\tau\in \overline{\Dis\setminus\Dis/2}$
and $\vartheta(\tau)=2\tau$ for all $\tau\in \Dis/2$. Consider now
the curve $\gamma_h:\Dis\to X$, defined as $\gamma_h(\tau):=\varrho_h(\vartheta(\tau))$, for all $\tau\in\Dis$.
Then, $\gamma_h\in\Gamma$, $\gamma_h(\Dis)=\varrho_h(\vartheta(\Dis))=\varrho_h(\Dis)\subset S$ and, of course,
\begin{equation}
	\label{construct0-bis}
\sup_{t\in \Dis} f(\lambda_h;\gamma_h(t))\leq c(\lambda_h)+\lambda_0-\lambda_h.
\end{equation}
Then, by arguing exactly as in the proof of~\cite[Lemma 2.1(ii)]{JT} by $\boldsymbol{({\mathcal H}_3)}$, for all $\eps>0$, 
\begin{equation}
	\label{conseqparipari}
\sup_{t\in \Dis} f(\lambda_0;\gamma_h(t))\leq c(\lambda_0)+\eps
\end{equation}
for every $h\geq 1$ large enough. In view of assumption $\boldsymbol{({\mathcal H}_2)}$, 
there exists $\omega=\omega(\lambda_0)>0$ small enough that
$c(\lambda_0)-3\omega>a(\lambda_0)$. Let us set
\begin{equation}
	\label{Mdefin}
M_h:=\overline{(f(\lambda_0;\cdot)\circ \gamma_h)^{-1}([c(\lambda_0)-3\omega,c(\lambda_0)+\omega])}.
\end{equation}
Therefore, $M_h\subset \Dis$ is of course closed and nonempty 
(just take $\eps=\omega$ in \eqref{conseqparipari} and use the definition of $c(\lambda_0)$) 
for $h\geq \bar h$, for some $\bar h=\bar h(\omega)\geq 1$. Moreover, $M_h\cap \Sf=\emptyset$ for all
$h\geq \bar h$. In fact, assume by contradiction
that, for some $h\geq \bar h$, there exists $\tau_h\in M_h\cap \Sf$. In turn, by definition, 
there exists a sequence $\xi_j^h\subset\Dis$ with $\xi_j^h\to \tau_h\in \Sf$ as $j\to\infty$ and 
$$
c(\lambda_0)-3\omega\leq f(\lambda_0;\varrho_h(\vartheta(\xi_j^h)))\leq c(\lambda_0)+\omega,
$$
for all $j\geq 1$. Then, noticing that $\vartheta(\xi_j^h)\in \Sf$ for $j\geq 1$ 
sufficiently large by the definition of $\vartheta$, we can conclude that
$$
c(\lambda_0)-3\omega\leq f(\lambda_0;\varrho_h(\vartheta(\xi_j^h)))\leq 
\sup_{\tau\in \Sf} f(\lambda_0;\varrho_h(\tau))\leq
a(\lambda_0)<c(\lambda_0)-3\omega,
$$
yielding the desired contradiction.
Then, on account of Proposition~\ref{approxxres}, for every $h\geq \bar h$, there exists a curve 
$\tilde\gamma_h\in C(\Dis,X)$ with $\tilde\gamma_h(\Dis)\subset S$ such that $\|\tilde\gamma_h(t)-\gamma_h(t)^*\|_V\leq \lambda_0-\lambda_h$,
for all $t\in M_h$ and $\tilde\gamma_h(\tau)=\gamma_h(\tau)^{H_0(h)}$, for all $\tau\in \Sf$.
In particular~\eqref{closssdd} holds. Furthermore, it is $\tilde\gamma_h\in\Gamma$, since
$$
\tilde\gamma_h|_{\Sf}=\gamma_h|_{\Sf}^{H_0(h)}=\varrho_h|_{\Sf}^{H_0(h)}\in\Gamma_0.
$$
Taking into account how $\tilde\gamma_h$ is constructed (by iterated 
polarizations, according to Lemma~\ref{approxxres}), 
by assumption $\boldsymbol{({\mathcal H}_4)}$ and inequality~\eqref{construct0-bis}, for all $h\geq \bar h$ we have
\begin{equation}
	\label{construct1}
\sup_{t\in \Dis} f(\lambda_h;\tilde\gamma_h(t))\leq
\sup_{t\in \Dis} f(\lambda_h;\gamma_h(t))\leq c(\lambda_h)+\lambda_0-\lambda_h.
\end{equation}
At this point, proceeding exactly as in the proof of~\cite[Lemma 2.1(i)]{JT} there exists 
a positive constant ${\mathcal M}={\mathcal M}(\lambda_0)$ such that implication~\eqref{prima} hold. Finally, by combining 
\eqref{conseqparipari} with $f(\lambda_0;\tilde\gamma_h(t))\leq f(\lambda_0;\gamma_h(t))$
(again in light of $\boldsymbol{({\mathcal H}_4)}$) it also follows that \eqref{seconda} holds.
\end{proof}

We can now proceed with the proof of the main result, Theorem~\ref{maintth1}.

\subsection{Proof of Theorem~\ref{maintth1}}
Fix an arbitrary $\lambda_0\in \Lambda$ such that condition~\eqref{dung}
is satisfied and let $(\lambda_h)$ be a related strictly increasing sequence converging to $\lambda_0$.
We know that the set $D\subseteq\Lambda$ of such values has full measure ${\mathcal L}^1(\Lambda)$.
According to Lemma~\ref{lemmaevolut}, there exist $\bar h\geq 1$ (depending upon $\lambda_0$), two sequences of paths $(\gamma_h)_{h\geq \bar h},
(\tilde\gamma_h)_{h\geq \bar h}\subset\Gamma$ with $\gamma_h(\Dis),\tilde\gamma_h(\Dis)\subset S$,
a sequence $(M_h)_{h\geq \bar h}$ of nonempty closed subsets of $\Dis$, disjoint from $\Sf$, and a positive constant
${\mathcal M}(\lambda_0)$ such that conditions~\eqref{closssdd}, \eqref{prima} and \eqref{seconda} hold. 
Let $\omega=\omega(\lambda_0)$ be the positive number which appears in the definition~\eqref{Mdefin} of $M_h$.
Then, for any fixed $\delta\in (0,\omega]$ small, there exists $h_\delta\geq \bar h$ such that
the following facts hold:
\begin{align}
	\label{secondabis}
& \sup_{t\in \Dis} f(\lambda_0;\tilde\gamma_{h_\delta}(t))\leq \sup_{t\in \Dis} 
f(\lambda_0;\gamma_{h_\delta}(t))\leq c(\lambda_0)+\delta,\qquad
0<\lambda_0-\lambda_{h_\delta}\leq \delta,  \\
	\label{primabis}
& f(\lambda_0;\tilde\gamma_{h_\delta}(t))\geq c(\lambda_0)-\lambda_0+\lambda_{h_\delta}
\quad
\Longrightarrow
\quad
\|\tilde\gamma_{h_\delta}(t)\|\leq {\mathcal M}(\lambda_0), \\
\noalign{\vskip3.5pt}
	\label{closssdd-spec}
& \|\tilde\gamma_{h_\delta}(t)-\gamma_{h_\delta}(t)^*\|_V\leq \delta,
\,\,\quad\text{for all $t\in M_{h_\delta}$}.
\end{align}
For all $\delta\in (0,\omega]$, we denote by $A_\delta$ the closed set defined as follows
$$
A_\delta:=\big\{u\in X: \|u\|\leq {\mathcal M}(\lambda_0),\,\, u\in \tilde\gamma_{h_\delta}(\Dis)
\cap f(\lambda_0;\cdot)^{-1}([c(\lambda_0)-2\delta,c(\lambda_0)+2\delta])\big\},
$$
and we set
$$
C_\delta:=\{u\in X:\, d(u,A_\delta)\leq \sqrt{\delta},\,\,  c(\lambda_0)-2\delta \leq f(\lambda_0,u)\leq c(\lambda_0)+2\delta\}.
$$
Since $f(\lambda_0;\cdot)$ is continuous, $C_\delta$ is of course closed in $X$.
We claim that $C_\delta\not=\emptyset$, for any $\delta\in (0,\omega]$. In fact, let 
$w_\delta:=\tilde\gamma_{h_\delta}(t_\delta)\in S$ with $t_\delta\in\Dis$, by continuity, such that
$$
\max_{t\in \Dis} f(\lambda_0;\tilde\gamma_{h_\delta}(t))=f(\lambda_0;w_\delta).
$$ 
Then, it follows that
$$
c(\lambda_0)-2\delta\leq c(\lambda_0)-\lambda_0+\lambda_{h_\delta}
\leq c(\lambda_0)\leq f(\lambda_0;w_\delta)\leq c(\lambda_0)+2\delta.
$$
This, by virtue of~\eqref{primabis}, also yields
$\|w_\delta\|=\|\tilde\gamma_{h_\delta}(t_\delta)\|\leq {\mathcal M}(\lambda_0)$. Hence $w_\delta\in A_\delta$
and, in turn, $w_\delta\in C_\delta$, proving the claim.
Given now $\delta\in (0,\omega]$, assume by contradiction that
\begin{equation}
	\label{mainimplicc}
\forall u\in X:\quad
u\in C_\delta		
\quad
\Longrightarrow
\quad 
|df(\lambda_0;\cdot)|(u)>2\sqrt{\delta}.
\end{equation}
By the Quantitative Deformation Lemma~\ref{theorlem} (applied to $f(\lambda_0;\cdot)$ with 
the choice $\sigma:=\sqrt{\delta}$), we can find a continuous 
map $\eta_\delta:X\times [0,1]\to X$ with the following properties: 
\begin{gather}
	\label{prim}
f(\lambda_0;\eta_\delta (u,t))\leq f(\lambda_0;u),\qquad  \|\eta_\delta(u,t)-u\|\leq \sqrt{\delta}t, \\
	\label{secon}
u\in A_\delta,\,\,  c(\lambda_0)-\delta \leq f(\lambda_0,u)\leq c(\lambda_0)+\delta
\,\,\,\Longrightarrow\,\,\,
f(\lambda_0;\eta_\delta(u,1))\leq c(\lambda_0)-\delta,
\end{gather}
for all $u\in X$ and $t\in[0,1]$. Let now $\Theta:X\to [0,1]$ be a continuous function such that
\begin{align*}
& \Theta(u)=0,\quad \text{for all $u\in C_1$},\qquad C_1:=\{u\in X: f(\lambda_0;u)\leq a(\lambda_0)\}, \\
&\Theta(u)=1,\quad \text{for all $u\in C_2$},\qquad C_2:=\{u\in X: f(\lambda_0;u)\geq c(\lambda_0)-\delta\}.
\end{align*}
Such a map exists since $C_1,C_2$ are nonempty closed subsets of $X$ 
and $C_1\cap C_2=\emptyset$. Then, we consider the curve 
$\hat\gamma:\Dis\to X$ defined by setting 
$$
\hat\gamma(t):=\eta_\delta(\tilde\gamma_{h_\delta}(t),\Theta(\tilde\gamma_{h_\delta}(t))),\quad \text{for all $t\in \Dis$}.
$$
Of course $\hat\gamma$ is continuous. 
Moreover, $\hat\gamma|_{\Sf}$ belongs to $\Gamma_0$. In fact, taken $\tau\in\Sf$, we have
$$
f(\lambda_0;\tilde\gamma_{h_\delta}(\tau))\leq \sup_{\gamma_0\in\Gamma_0}\sup_{\tau\in\Sf} f(\lambda_0;\gamma_0(\tau))=a(\lambda_0).
$$
Then, by the definition and properties of $\eta_\delta$ and $\Theta$, we have
$$
\hat\gamma(\tau)=\eta_\delta(\tilde\gamma_{h_\delta}(\tau),\Theta(\tilde\gamma_{h_\delta}(\tau)))
=\eta_\delta(\tilde\gamma_{h_\delta}(\tau),0)=\tilde\gamma_{h_\delta}(\tau),
\quad\text{for every $\tau\in\Sf$}.
$$
Thus $\hat\gamma$ belongs to $\Gamma$. Consider now an
arbitrary point $t\in \Dis$. If it is the case that
$$
f(\lambda_0;\tilde\gamma_{h_\delta}(t))\leq c(\lambda_0)-(\lambda_0-\lambda_{h_\delta}), 
$$
then by the first inequality in~\eqref{prim}, we have
\begin{equation}
	\label{concl1}
f(\lambda_0;\hat\gamma(t))\leq c(\lambda_0)-(\lambda_0-\lambda_{h_\delta}).
\end{equation}
On the contrary, in the case 
$$
f(\lambda_0;\tilde\gamma_{h_\delta}(t))> c(\lambda_0)-(\lambda_0-\lambda_{h_\delta})\geq c(\lambda_0)-\delta,
$$ 
it then follows by~\eqref{primabis} that $\|\tilde\gamma_{h_\delta}(t)\|\leq {\mathcal M}(\lambda_0)$, namely, on account of \eqref{secondabis}
$$
\tilde\gamma_{h_\delta}(t) \in A_\delta, \quad    c(\lambda_0)-\delta\leq f(\lambda_0;\tilde\gamma_{h_\delta}(t))\leq c(\lambda_0)+\delta,
$$ 
yielding, by virtue of implication~\eqref{secon} and the definition of $\Theta$,
\begin{equation}
	\label{concl2}
f(\lambda_0;\hat\gamma(t))= f(\lambda_0;\eta_\delta(\tilde\gamma_{h_\delta}(t),1))\leq c(\lambda_0)-\delta
\leq c(\lambda_0)-(\lambda_0-\lambda_{h_\delta}).
\end{equation}
Hence, by combining inequalities~\eqref{concl1}-\eqref{concl2}, we conclude that
$$
c(\lambda_0)\leq \sup_{t\in [0,1]}f(\lambda_0;\hat\gamma(t))
\leq c(\lambda_0)-(\lambda_0-\lambda_{h_\delta})<c(\lambda_0),
$$
namely the desired contradiction. Therefore, by choosing
$\delta=1/j$, there exists a sequence $(u_j)\subset X$ ($u_j\in C_j$), contained
in the ball centered at the origin and of radius ${\mathcal M}(\lambda_0)+2$, 
such that $f(\lambda_0;u_j)\to c(\lambda_0)$, as $j\to\infty$,
and $|df(\lambda_0;\cdot)|(u_j)\to 0$, as $j\to\infty$. At this stage, we have proved that $f(\lambda_0;\cdot)$ 
admits a bounded Palais-Smale sequence at the Mountain Pass value $c(\lambda_0)$. Let now $A_j$, $M_j$, 
$\gamma_j$ and $\tilde\gamma_j$ denote $A_{\delta}$, $M_{h_{\delta}}$, $\gamma_{h_{\delta}}$ and $\tilde\gamma_{h_{\delta}}$
respectively, with $\delta=1/j$ for $j\geq 1/\omega$. We claim that $A_j\subset \tilde\gamma_j(M_j)$. If $y\in A_j$, there
exists $\tau\in\Dis$ with $y=\tilde\gamma_j(\tau)$ and 
$c(\lambda_0)-2/j\leq f(\lambda_0;\tilde\gamma_j(\tau))\leq c(\lambda_0)+2/j$, 
yielding, by $\boldsymbol{({\mathcal H}_4)}$ and\eqref{secondabis},
$$
c(\lambda_0)-3\omega\leq c(\lambda_0)-2/j\leq f(\lambda_0;\tilde\gamma_j(\tau))\leq 
f(\lambda_0;\gamma_j(\tau))\leq c(\lambda_0)+1/j\leq c(\lambda_0)+\omega.
$$
Hence, $\tau\in (f(\lambda_0;\cdot)\circ \gamma_{j})^{-1}([c(\lambda_0)-3\omega,c(\lambda_0)+\omega])\subset M_j$,
namely $y\in \tilde\gamma_{j}(M_j)$, proving the claim. Hence, from $d(u_j,A_j)\leq 1/\sqrt{j}$ (recall that $u_j\in C_j$), we deduce
\begin{equation}
	\label{controlsimtry}
d(u_j,\tilde\gamma_{j}(M_j))\leq {1}/{\sqrt{j}}.
\end{equation}
According to Section~\ref{abstractsymmetrrr}, $u^*_j$ is defined. Moreover,
for all $\tau\in M_j$, since $\tilde\gamma_{j}(\tau)^*=\gamma_{j}(\tau)^*$ by construction
and (3) of framework~\ref{abstractsymmetrrr}, we have $\|\gamma_{j}(\tau)^*-u^*_j\|_V\leq 
C_\Theta\|\tilde\gamma_{j}(\tau)-u_j\|_V$, by inequality~\eqref{misspropp}. Then, for some constant $C$, 
on account of~\eqref{closssdd-spec} and~\eqref{controlsimtry}, 
\begin{align*}
\|u_j-u^*_j\|_V &\leq\inf_{\tau\in M_j}\big[
\|u_j-\tilde\gamma_{j}(\tau)\|_V+
\|\tilde\gamma_{j}(\tau)-\gamma_{j}(\tau)^*\|_V+
\|\gamma_{j}(\tau)^*-u^*_j\|_V \big] \\
&\leq\inf_{\tau\in M_j}\big[
(1+C_\Theta)K\|u_j-\tilde\gamma_{j}(\tau)\|+
\|\tilde\gamma_{j}(\tau)-\gamma_{j}(\tau)^*\|_V \big]\leq C/\sqrt{j},
\end{align*}
where $K>0$ is the continuity constant of $X\hookrightarrow V$. This concludes the proof.
\qed

\subsection{Proof of Corollary~\ref{maintth1-cor}}
Let $\lambda_0\in\Lambda$ such that there exists a $(SBPS)_{c(\lambda_0)}$-sequence $(u_j)\subset X$.
Since $f(\lambda_0;\cdot)$ satisfies $(SBPS)_{c(\lambda_0)}$,
there exists a subsequence $(u_{j_m})$ of $(u_j)$ which converges 
to some $u$ in $X$. By Proposition~\ref{lowersecslope}, we have $|df(\lambda_0;\cdot)|(u)=0$. By continuity,
$f(\lambda_0;u)=c(\lambda_0)$. Recalling~\eqref{misspropp}, 
\begin{align}
	\label{symcomp}
\|u-u^*\|_V &\leq \lim_{j\to\infty} (\|u-u_{j_m}\|_V+\|u_{j_m}-u^*_{j_m}\|_V+\|u^*_{j_m}-u^*\|_V) \\
& \leq \lim_{j\to\infty} ((1+C_\Theta)K\|u-u_{j_m}\|+\|u_{j_m}-u^*_{j_m}\|_V)=0,  \notag
\end{align}
yielding $u=u^*$, as desired.
\qed

\subsection{Proof of Corollary~\ref{maincor2}}
There exists a strictly increasing sequence $(\lambda_j)\subset [1-\sigma,1]$ 
converging to $1$ such that, for each $j\geq 1$, the functional $f(\lambda_j;\cdot)$
admits a Symmetric Bounded Palais-Smale sequence $(u^j_m)$ at the Mountain Pass energy level $c(\lambda_j)$, namely
\begin{equation*}
	\lim_m f(\lambda_j;u^j_m)=c(\lambda_j),\qquad  \lim_m |df(\lambda_j;\cdot)|(u^j_m)= 0,
	\qquad \lim_m\|u^j_m-u^{j*}_m\|_V= 0.
\end{equation*}
Since $f(\lambda_j;\cdot)$ satisfies $(SBPS)_{c(\lambda_j)}$, for all $j\geq 1$
there exists a subsequence $(u^j_{m_k})$ of $(u^j_m)$ such that $u^j_{m_k}\to u_j$ in $X$, as $k\to\infty$.
Recalling Proposition~\ref{lowersecslope}, we see that properties \eqref{siamo1} hold. Notice that the 
symmetry conclusion follows again as in~\eqref{symcomp}. \qed

\vskip30pt
\noindent
{\bf Acknowledgments.} 
The author wishes to thank Louis Jeanjean for providing some very useful suggestions
and Daniele Bartolucci for pointing out some bibliographic references.

\bigskip
\bigskip

\end{document}